\pdfoutput=1
\documentclass[letterpaper]{amsart}

\usepackage[utf8]{inputenc}
\usepackage{lmodern}
\usepackage[T1]{fontenc}
\usepackage{amssymb}
\usepackage{enumitem}
\usepackage{mathtools}

\usepackage{tikz-cd}
\usepackage[pdfusetitle]{hyperref}
\usepackage{cleveref}

\tikzcdset{arrow style=Latin Modern}

\newlist{enumarabic}{enumerate}{1}
\setlist[enumarabic]{font=\normalfont,label=(\arabic*),leftmargin=0.3in}
\newlist{enumroman}{enumerate}{1}
\setlist[enumroman]{font=\normalfont,label=(\roman*),leftmargin=0.3in}

\numberwithin{equation}{section}
\allowdisplaybreaks[4]

\theoremstyle{plain}
\newtheorem{theorem}{Theorem}[section]
\newtheorem{proposition}[theorem]{Proposition}
\newtheorem{lemma}[theorem]{Lemma}
\newtheorem{corollary}[theorem]{Corollary}

\theoremstyle{definition}

\newtheorem{remark}[theorem]{Remark}
\newtheorem{example}[theorem]{Example}

\theoremstyle{remark}
\newtheorem*{acknowledgements}{Acknowledgements}

\let\newterm\emph

\def\arxiv#1{\href{http://arxiv.org/abs/#1}{\texttt{arXiv:#1}}}
\def\doi#1{\href{http://doi.org/#1}{doi:#1}}

\def\cf{\emph{cf.}}

\DeclareMathAlphabet\mathbfit{OML}{cmm}{b}{it}

\let\epsilon\varepsilon
\let\phi\varphi
\let\emptyset\varnothing

\def\N{\mathbb N}
\def\Z{\mathbb Z}

\def\R{\mathbb R}
\def\C{\mathbb C}

\def\RP{\mathbb{RP}}

\DeclareMathOperator{\Tor}{Tor}
\DeclareMathOperator{\supp}{supp}
\DeclareMathOperator*{\colim}{colim}
\def\deg#1{|#1|}
\def\pair#1{\langle#1\rangle}
\def\bigpair#1{\bigl\langle#1\bigr\rangle}

\def\XXX{\mathbfit{X}}
\def\ZZZ{\mathbfit{Z}}

\let\shuffle\nabla

\def\ZK{\mathcal{Z}}
\def\kk{\Bbbk}

\def\SRC#1{\kk\langle#1\rangle}
\def\FU#1{\mathcal{F}(#1)}
\def\restr#1#2{#1_{|#2}}

\def\CC{C}
\def\cc{o}

\def\DJ{DJ}

\def\KK{\mathbf{K}}

\def\AW{AW}

\def\BB{\tilde{A}}

\def\bG{\bar G}
\def\tN{\tilde N}

\def\hsigma{\hat\sigma}
\def\hSigma{\hat\Sigma}
\def\hN{\hat N}
\def\tXXX{\boldsymbol{\tilde X}}
\def\tYYY{\boldsymbol{\tilde Y}}

\def\Ri{\R^{\times}}
\def\GGG{\mathbfit{G}}
\def\KKK{\mathbfit{K}}
\def\LLL{\mathbfit{L}}
\def\TTT{\mathbfit{T}}
\def\YYY{\mathbfit{Y}}
\def\ZZZ{\mathbfit{Z}}
\def\bGGG{\boldsymbol{\bar G}}
\def\tKKK{\boldsymbol{\tilde K}}
\def\tLLL{\boldsymbol{\tilde L}}
\def\tLLL{\vphantom{\LLL}\smash{\boldsymbol{\tilde L}}}
\def\GGGG{\mathcal{G}}
\def\KKKK{\mathcal{K}}
\def\LLLL{\mathcal{L}}

\def\GG{G}
\def\posub#1{\bar{o}_{#1}}
\def\posup#1{\bar{o}^{\,#1}}

\def\timesunder#1{\mathbin{\mathchoice
  {\mathop\times\limits_{\mkern-5mu #1\mkern-5mu}}%
  {\times_{#1}}{\times_{#1}}{\times_{#1}}}}
\def\timesover#1{\mathbin{\mathchoice
  {\mathop\times\limits^{\mkern-3mu #1\mkern-5mu}}%
  {\times^{#1}}{\times^{#1}}{\times^{#1}}}}

\def\nn#1{\bar#1}

\begin{document}

\title[The cohomology rings of real toric spaces]{The cohomology rings of real toric spaces\\and smooth real toric varieties}
\author{Matthias Franz}
\thanks{The author was supported by an NSERC Discovery Grant.}
\address{Department of Mathematics, University of Western Ontario,
  London, Ont.\ N6A\;5B7, Canada}
\email{mfranz@uwo.ca}

\hypersetup{pdfauthor=\authors,pdftitle={The cohomology rings of real toric spaces and smooth real toric varieties}}


\begin{abstract}  
  We compute the cohomology rings of smooth real toric varieties and of real toric spaces,
  which are quotients of real moment-angle complexes by freely acting subgroups of the ambient \(2\)-torus.
  The differential graded algebra we present is in fact an equivariant dga model, valid for arbitrary coefficients.
  We deduce from our description that smooth toric varieties are M-varieties.
\end{abstract}

\maketitle

\section{Introduction}

Let \(\Sigma\) be a simplicial complex on the vertex set~\([m]=\{1,\dots,m\}\), possibly with ghost vertices.
The \(2\)-torus~\(G=(\Z_{2})^{m}\) acts canonically on the real moment-angle complex
\begin{equation}
  Z_{\Sigma} = \ZK_{\Sigma}(D^{1},S^{0}) = \bigcup_{\sigma\in\Sigma} Z_{\sigma} \subset (D^{1})^{m}
\end{equation}
where \(D^{1}=[-1,1]\) is the interval with boundary~\(S^{0}=\{\pm1\}\), and the exponents in
\begin{equation}
  Z_{\sigma} = (D^{1})^{\sigma}\times (S^{0})^{[m]\setminus\sigma}
\end{equation}
indicate the factors of the Cartesian products. The quotient~\(X_{\Sigma} = Z_{\Sigma}/K\)
by a freely acting subgroup~\(K\subset G\) is called a \newterm{real toric space} \cite{ChoiKajiTheriault:2017},~\cite{ChoiPark:2017a}.
It is the topological counterpart of a smooth real toric variety.

More precisely, let \(\Sigma\) be a regular rational fan in~\(\R^{n}\) with \(m\)~rays,
which we may also consider as a simplicial complex on the vertex set~\([m]\).
The associated real toric variety~\(\XXX_{\Sigma}=\XXX_{\Sigma}(\R)\)
is the fixed point set of the canonically defined complex conjugation
on the complex toric variety~\(\XXX_{\Sigma}(\C)\).
If the primitive generators of the rays in~\(\Sigma\) span the lattice~\(\Z^{n}\subset\R^{n}\),
then there is a real toric space~\(X_{\Sigma}\) as defined above
that is a strong deformation retract of the real toric variety~\(\XXX_{\Sigma}\).
For an arbitrary regular fan this statement has to be modified slightly, see~\Cref{thm:rts-toric-retraction}.

Jurkiewicz~\cite{Jurkiewicz:1985} has determined the mod~\(2\) cohomology rings of smooth real projective toric varieties,
and Davis--Januszkiewicz~\cite{DavisJanuszkiewicz:1991} those of their topological analogues, small covers (over a polytope),
compare \Cref{rem:equivformal}.
The rational Betti numbers of these spaces have been computed by Suciu--Trevisan~\cite{Suciu:2013}, \cite[Sec.~VI.2]{Trevisan:2012},
and their integral cohomology groups by Cai--Choi~\cite[Thm.~1.1]{CaiChoi} (more generally for shellable~\(\Sigma\)).
The cohomology ring of a real moment-angle complex has been calculated by Cai~\cite[p.~514]{Cai:2017}
for any coefficient ring~\(\kk\).
Choi--Park~\cite[Thm.~4.5]{ChoiPark:2017},~\cite[Main Thm.]{ChoiPark:2017a} have determined
the cohomology ring of any real toric space under the assumption that \(2\)~is invertible in~\(\kk\).
The special role of the prime~\(p=2\) was already pointed out by Choi--Kaji--Theriault~\cite[p.~155]{ChoiKajiTheriault:2017}
in their study of the \(p\)-local homotopy type of real toric spaces.

In this note we describe
the cohomology rings of real toric spaces and smooth real toric varieties with arbitrary coefficients.
Apart from generalizing the results mentioned above,
this also complements the author's computation~\cite{Franz:torprod} of the cohomology rings
of smooth (complex) toric varieties as well as those of quotients of (complex) moment-angle complexes
by freely acting closed subgroups of the ambient torus~\(T=(S^{1})^{m}\).

For any commutative ring~\(\kk\) with unit
we write \(\kk[\Sigma]\) for the anti-commutative face algebra of~\(\Sigma\). It is generated
by indeterminates~\(t_{1}\),~\dots,~\(t_{m}\) of degree~\(1\) with relations
\begin{equation}
  t_{j}\,t_{j'} = -t_{j'}\,t_{j}
  \qquad\text{and}\qquad
  t_{j_{1}}\cdots t_{j_{k}} = 0
\end{equation}
for~\(j\ne j'\) and any non-face~\(\{j_{1},\dots,j_{k}\}\notin\Sigma\).

Fixing an isomorphism~\(L\cong(\Z_{2})^{n}\), we can encode the quotient map~\(G\to L\)
by an \((n\times m)\)-matrix~\((x_{ij})\) with coefficients in~\(\Z_{2}\).
This gives well-defined coordinate functions~\(s_{1},\dots,s_{n}\colon L\to\{0,1\}\subset\kk\),
which generate the algebra~\(\FU{L}\) of all function~\(L\to\kk\) as a unital \(\kk\)-algebra.

We define a differential graded algebra (dga) structure on the tensor product
\begin{equation}
  A(\Sigma,L) = \kk[\Sigma]\otimes_{\kk}\FU{L}
\end{equation}
as follows: Both \(\kk[\Sigma]\) and~\(\FU{L}\) are subalgebras.
There are the additional relations
\begin{equation}
  \label{eq:rel-si-tj}
  s_{i}\,t_{j} =
  \begin{cases}
    t_{j}\,s_{i} & \text{if \(x_{ij}=0\),}\\
    t_{j}\,(1-s_{i}) & \text{if \(x_{ij}=1\)}
  \end{cases}
\end{equation}
for all~\(i\) and~\(j\),
and the differential is determined by
\begin{equation}
  \label{eq:intro:diff-si-tj}
  d\,t_{j} = -2\,t_{j}^{2}
  \qquad\text{and}\qquad
  d\,f = \sum_{j=1}^{m} (f\,t_{j}-t_{j}\,f)
\end{equation}
for~\(f\in\FU{L}\).
The group~\(L\) acts on~\(A(\Sigma,L)\) via its action on functions~\(L\to\kk\).

Note that for~\(\kk=\Z_{2}\) we recover the usual Stanley--Reisner ring of~\(\Sigma\)
with commuting generators of degree~\(1\) and no differential. The remaining relations become
\begin{equation}
  \label{eq:rel-mod2}
  s_{i}\,t_{j}=t_{j}\,s_{i} + x_{ij}\,t_{j}
  \qquad\text{and}\qquad
  d\,s_{i} = \sum_{j=1}^{m} x_{ij}\,t_{j}.
\end{equation}

\begin{theorem}
  \label{thm:main}
  The dga~\(A(\Sigma,L)\) is quasi-isomorphic to the singular cochain algebra~\(C^{*}(X_{\Sigma};\kk)\).
  The quasi-iso\-mor\-phism is \(L\)-equivariant and natural with respect to the inclusion of subcomplexes of~\(\Sigma\).
  In particular, there is an isomorphism of graded algebras
  \begin{equation*}
    H^{*}(X_{\Sigma};\kk) = H^{*}(A(\Sigma,L)).
  \end{equation*}
\end{theorem}

The analogous statement holds for all smooth real toric varieties (\Cref{thm:model-toric}),
and it implies that for~\(\kk=\Z_{2}\) the cohomology of these spaces can additively be expressed as a torsion product
involving \(\kk[\Sigma]\), see \Cref{thm:iso-Tor}.

\medbreak

Recall that a complex variety~\(\XXX\) with an antiholomorphic involution~\(\tau\) is called maximal
(or an M-variety) if the mod~\(2\) Betti sum of the real locus~\(\XXX^{\tau}\) is as large as allowed by Smith theory,
that is, equal to the corresponding sum for~\(\XXX\), compare~\cite[Prop.~7.3.7]{Hausmann:2014}.
This generalizes the classical notion of M-curves introduced by Harnack.
Bihan--Franz--McCrory--van Hamel~\cite[Thm.~1.1]{BihanEtAl:2006} have shown
that a toric variety is maximal if it is smooth and compact,
or (if Borel--Moore homology is used) possibly singular and of dimension at most~\(3\).
As a consequence of our results we obtain the following substantial enlargement of this list,
see~\Cref{thm:maximal}.

\begin{theorem}
  Every smooth toric variety is maximal.
\end{theorem}

The example of a \(6\)-dimensional non-maximal singular projective toric variety due to Hower~\cite{Hower:2008}
shows that the smoothness assumption cannot be left out.

\smallskip

The paper is organized as follows. In \Cref{sec:prelim} we review necessary background material
including anti-commutative face (co)algebras.
Real toric spaces are considered in \Cref{sec:rts} and toric varieties in \Cref{sec:toric}.

\begin{acknowledgements}
  The first version of this paper only considered the case~\(\kk=\Z_{2}\).
  I thank Xin Fu for encouraging me to think about the integer case.
  The resulting generalization to arbitrary coefficients has also led to a simpler proof.
\end{acknowledgements}

\section{Preliminaries}
\label{sec:prelim}

\subsection{Complexes and (co)algebras}

We work over a commutative ring~\(\kk\) with unit.
Unless indicated otherwise, all tensor products are taken over~\(\kk\),
and all (co)chains complexes and (co)homologies are with coefficients in~\(\kk\).

The differential on the dual~\(C^{*}\) of a complex~\(C\) is defined by
\begin{equation}
  \label{eq:sign-diff-dual}
  \pair{d\,\gamma,c} = -(-1)^{\deg{\gamma}}\,\pair{\gamma,d\,c}
\end{equation}
for~\(\gamma\in C^{*}\) and~\(c\in C\). Given another complex~\(B\), the tensor products~\(B\otimes C\) and~\(B^{*}\otimes C^{*}\)
pair via the formula
\begin{equation}
  \label{eq:sign-pairing-tensor}
  \bigpair{\beta\otimes\gamma,b\otimes c} = (-1)^{\deg{\gamma}\deg{b}}\,\pair{\beta,b}\,\pair{\gamma,c}
\end{equation}
for~\(\beta\in B^{*}\),~\(b\in B\) and \(\gamma\),~\(c\) as before.

We abbreviate `differential graded (co)algebra' to dgc and dga, respectively.
A \(\GG\)-complex, \(\GG\)-dgc or \(\GG\)-dga has a linear action of the group~\(\GG\) and equivariant structure maps.
Note that the quotient~\(C/K\) of a \(\GG\)-dgc~\(C\) by a normal subgroup~\(K\lhd\GG\) is a \(\GG/K\)-dgc.
A \(\GG\)-homotopy equivalence between \(\GG\)-complexes is a homotopy equivalence for which all maps involved are \(\GG\)-equivariant.

\subsection{Normalized chains and cochains}

It will be convenient to use simplicial sets. A good reference for our purposes is \cite{May:1968}.
All simplicial sets we consider will be Kan complexes.

The normalized chain complex~\(C(X)\) of a simplicial set~\(X\) is obtained
from non-normalized chains by dividing out the subcomplex of degenerate chains,
\cf~\cite[Sec.~VIII.6]{MacLane:1967}. The projection map is a homotopy equivalence.
For~\(X=pt\) a point we have a canonical isomorphism~\(C(pt)=\kk\).

The complex~\(C(X)\) is a dgc with the Alexander--Whitney diagonal given by
\begin{equation}
  \Delta\,x = \sum_{k=0}^{n} x(0,\dots,k) \otimes x(k,\dots,n)
\end{equation}
for an \(n\)-simplex~\(x\in X\), where the terms on the right indicate faces of~\(x\).
Its dual is the dga~\(C^{*}(X)\) of normalized singular cochains on~\(X\)
with the usual cup product. If \(X\) is discrete, then \(C^{*}(X)=H^{*}(X)\) is the algebra~\(\FU{X}\)
of functions~\(X\to\kk\) with pointwise operations.

If a (discrete) group~\(G\) acts on~\(X\) from the left,
then \(C(X)\) becomes a left \(G\)-dgc and \(C^{*}(X)\) a right \(G\)-dga.

We recall the Eilenberg--Zilber maps:
Given two simplicial sets~\(X\) and~\(Y\), there are the shuffle map
\begin{align}
  \shuffle=\shuffle_{X,Y}\colon& C(X)\otimes C(Y) \to C(X\times Y), \\
  \shortintertext{the Alexander--Whitney map}
  \AW=\AW_{X,Y}\colon& C(X\times Y) \to C(X)\otimes C(Y) \\
  \shortintertext{and the Eilenberg--Mac\,Lane homotopy}
  H=H_{X,Y}\colon& C(X\times Y) \to C(X\times Y),
\end{align}
all natural in the pair~\((X,Y)\). They form a contraction in the sense that
\begin{equation}
  \AW\,\shuffle = 1
  \qquad\text{and}\qquad
  d(H) = d\,H+H\,d = \shuffle\,\AW - 1,
\end{equation}
see~\cite[Thm.~2.1a]{EilenbergMacLane:1954}.
Here we have written ``\(1\)'' for the identity maps on~\(C(X)\otimes C(Y)\) and~\(C(X\times Y)\).
The shuffle map and the Alexander--Whitney map are associative, and the shuffle map is
additionally a morphism of dgcs \cite[(17.6)]{EilenbergMoore:1966}.

By iteration, the Eilenberg--Zilber contraction extends to more than two factors.
For example, in the presence of a third simplicial set~\(Z\) one uses the homotopy
\begin{equation}
  H_{X,Y,Z} = \shuffle_{X\times Y,Z}\,(H_{X,Y}\otimes 1)\,\AW_{X\times Y,Z} + H_{X\times Y,Z}
\end{equation}
to compare the shuffle map and the Alexander--Whitney map
relating the two dgcs~\(C(X)\otimes C(Y)\otimes C(Z)\) and~\(C(X\times Y\times Z)\).

\subsection{Classifying spaces and universal bundles}
\label{sec:universal}

Given a (discrete) group~\(\GG\), there is the simplicial version~\(E\GG\to B\GG\) of the universal right \(\GG\)-bundle.
See~\cite[Sec.~3.7]{Franz:2003} for the more general case of a simplicial group 
or also~\cite[\S 21]{May:1968} for a version where the group acts on the total space from the left.
The \(n\)-simplices of~\(EG\) are of the form~\(e = ([g_{0},\dots,g_{n-1}],g_{n})\)
with~\(g_{0}\),~\dots,~\(g_{n}\in G\). Such a simplex is degenerate if and only if \(g_{i}=1\) for some~\(i<n\).
For any~\(0\le i\le j\le n\) the face of~\(e\) with vertices~\(i\),~\(i+1\),~\dots,~\(j\) is given by
\begin{equation}
  e(i,\dots,j) = \bigl([g_{i},\dots,g_{j-1}],g_{j}\cdots g_{n}\bigr),
\end{equation}
and for~\(n\ge 1\) the differential of~\(e\) is
\begin{align}
  d\,e &= \bigl([g_{1},\dots,g_{n-1}],g_{n}\bigr)
  + \sum_{i=1}^{n-1}(-1)^{i}\,\bigl([g_{0},\dots,g_{i-1}\,g_{i},\dots,g_{n-1}],g_{n}\bigr) \\
  \notag &\qquad + (-1)^{n}\,\bigl([g_{0},\dots,g_{n-2}],g_{n-1}\,g_{n}\bigr).
\end{align}

In our application, \(G\) will be a \(2\)-torus.
The case where \(\GG=\CC=\{1,o\}\) is the group with two elements will therefore be particularly important.
In this situation \(E\CC\) has exactly two non-degenerate simplices in each dimension~\(n\ge0\), namely
\begin{equation}
  w_{n} = ([\underbrace{o,\dots,o}_{n}],1)
  \qquad\text{and}\qquad
  w_{n}\,o = ([\underbrace{o,\dots,o}_{n}],o).
\end{equation}
In the normalized chain complex~\(C(E\CC)\) one has
\begin{equation}
  \label{eq:diff-EH}
  d\,w_{n} = w_{n-1} + (-1)^{n}\,w_{n-1}\,o
\end{equation}
for~\(n\ge1\) as well as
\begin{equation}
  \label{eq:Delta-EH}
  \Delta w_{n} = \sum_{i+j=n} w_{i}\,o^{j} \otimes w_{j}
\end{equation}
for~\(n\ge0\). Hence the image~\(\bar w_{n}=[o,\dots,o]\) of~\(w_{n}\) and \(w_{n}\,o\) in~\(C(B\CC)\) satisfies
\begin{equation}
  \label{eq:diff-Delta-BH}
  d\,\bar w_{n} = \begin{cases}
    2\,\bar w_{n} & \text{if \(n\ge2\) is even,} \\
    0 & \text{otherwise}
  \end{cases}
  \qquad\text{and}\qquad
  \Delta\,\bar w_{n} = \sum_{i+j=n} \bar w_{i} \otimes \bar w_{j}
\end{equation}
for all~\(n\ge0\).

Any simplicial set~\(X\) with a left action of a group~\(G\) gives rise to an associated bundle
\begin{equation}
  EG\timesunder{G}X \to BG \,;
\end{equation}
its total space is the simplicial Borel construction~\(X_{G}\) of~\(X\).
We similarly write the pull-back of~\(EG\to BG\) along a map~\(q\colon Y\to BG\) as
\begin{equation}
  Y\timesover{BG}EG.
\end{equation}
It has a canonical right \(G\)-action that we convert to a left action in the usual way.

\begin{lemma}
  \label{eq:equiv-koszul}
  Let \(\GG\) be a group, and let \(X\) be a simplicial left \(\GG\)-space.
  There is a natural map
  \begin{equation*}
    X_{\GG}\timesover{B\GG} E\GG = \bigl( E\GG\timesunder{\GG}X \bigr) \timesover{B\GG} E\GG \to X
  \end{equation*}
  that is a homotopy equivalence and also \(G\)-equivariant.
\end{lemma}

\begin{proof}
  See~\cite[Thm.~2.8.2]{Franz:2001} or~\cite[Thm.~3.11]{Franz:2003}.
\end{proof}

\subsection{Anti-commutative face (co)algebras}
\label{sec:face}

Let \(\Sigma\) be a simplicial complex on the vertex set~\([m]\).
For any~\(\sigma\in\Sigma\) we define the dgc
\begin{equation}
  \label{eq:def-SRC-sigma}
  \SRC{\sigma} = C(BC)^{\otimes\sigma},
\end{equation}
where \(C\) again denotes the group with two elements,
and based on this the dgc
\begin{equation}
  \label{eq:def-SRC-Sigma}
  \SRC{\Sigma} = \colim_{\sigma\in\Sigma} \SRC{\sigma}.
\end{equation}
We call \(\SRC{\Sigma}\) the anti-commutative face coalgebra of~\(\Sigma\),
compare~\cite[p.~335]{BuchstaberPanov:2015} for the commutative case (without differential).
A \(\kk\)-basis for~\(\SRC{\Sigma}\) is given by the elements~\(u_{\alpha}\)
for all~\(\alpha\in\N^{m}\) whose support satisfies
\begin{equation}
  \supp\alpha \coloneqq \bigl\{ j\in[m] \bigm| \alpha_{j}\ne0 \bigr\} \in\Sigma.
\end{equation}
Taking the identities~\eqref{eq:diff-Delta-BH} as well as the usual Koszul sign rule into account,
we see that on such a basis element~\(u_{\alpha}\) the diagonal is given by
\begin{equation}
  \label{eq:Delta-SRC}
  \Delta\,u_{\alpha} = \sum_{\beta+\gamma=\alpha} (-1)^{\epsilon(\beta,\gamma)}\,u_{\beta} \otimes u_{\gamma}
\end{equation}
where
\begin{equation}
  \label{eq:def-epsilon-beta-gamma}
  \epsilon(\beta,\gamma) =  \sum_{j<j'}\gamma_{j}\beta_{j'},
\end{equation}
and the differential by
\begin{equation}
  \label{eq:diff-SRC}
  d\,u_{\alpha} = 2 \!\!\! \sum_{\alpha_{j}>0\text{~even}} \!\!\! (-1)^{\epsilon(\alpha,j-1)}\,u_{\alpha|j}.
\end{equation}
Here we have written \(\alpha|j\in\N^{m}\) for the multi-index that is
obtained from~\(\alpha\) by decreasing the \(j\)-th coordinate by~\(1\), and
\begin{equation}
  \label{eq:def-epsilon-alpha-j}
  \epsilon(\alpha,j) = \alpha_{1}+\dots+\alpha_{j}.
\end{equation}

The anti-commutative face algebra~\(\kk[\Sigma]\) of~\(\Sigma\) is the dga dual to the dgc~\(\SRC{\Sigma}\).
We write \((t^{\alpha})\) for the \(\kk\)-basis dual to the basis~\((u_{\alpha})\) for~\(\SRC{\Sigma}\).
For convenience we also write \(u_{j}=u_{\beta(j)}\) and~\(t_{j}=t^{\beta(j)}\) for~\(1\le j\le m\)
where \(\beta(j)=(0,\dots,0,1,0,\dots,0)\) is the \(j\)-th canonical generator of~\(\N^{m}\).

The dual differential on~\(\kk[\Sigma]\) then is determined by
\begin{equation}
  \label{eq:diff-tj}
  d\,t_{j} = -2\,t_{j}^{2}.
\end{equation}
In fact, taking our sign rules~\eqref{eq:sign-diff-dual} and~\eqref{eq:sign-pairing-tensor} into account, we find
\begin{align}
  \bigpair{d\,t_{j},u_{2j}} &= \pair{t_{j},d\,u_{2j}} = \pair{t_{j},2\,u_{j}} = 2 \\
  \notag &= 2\,\pair{t_{j}, u_{j}} \, \pair{t_{j},u_{j}}
  = -2\,\pair{t_{j}\otimes t_{j}, u_{j}\otimes u_{j}} \\
  \notag &= -2\,\pair{t_{j}\otimes t_{j},\Delta\,u_{2j}} = \bigpair{-2\,t_{j}^{2},u_{2j}},
\end{align}
and both \(d\,t_{j}\) and~\(t_{j}^{2}\) vanish on any~\(u_{\alpha}\) with~\(\alpha\ne2\beta(j)\).

For~\(j<j'\) and~\(\alpha=\beta(j)+\beta(j')\) one similarly gets
\begin{align}
  \pair{t_{j}\,t_{j'},u_{\alpha}} &= \bigpair{t_{j}\otimes t_{j'},\Delta\,u_{\alpha}}
  =  \bigpair{t_{j}\otimes t_{j'}, u_{j}\otimes u_{j'}} = 1, \\
  \pair{t_{j'}\,t_{j},u_{\alpha}} &= \bigpair{t_{j'}\otimes t_{j},\Delta\,u_{\alpha}}
  =  -\bigpair{t_{j'}\otimes t_{j}, u_{j'}\otimes u_{j}} = -1.
\end{align}
Both \(t_{j}\,t_{j'}\) and~\(t_{j'}\,t_{j}\) vanish on all~\(u_{\gamma}\) with~\(\gamma\ne\alpha\),
which shows
\begin{equation}
  \label{eq:tj-tjp}
  t_{j}\,t_{j'} = - t_{j'}\,t_{j}
\end{equation}
for all~\(j\ne j'\).

Iterating the diagonal on~\(\SRC{\Sigma}\) gives
\begin{equation}
  \pair{t_{j_{1}}\cdots t_{j_{k}},u_{\alpha}} = 0
\end{equation}
for any non-face~\(\{j_{1},\dots,j_{k}\}\notin\Sigma\) and all~\(\alpha\) with~\(\supp\alpha\in\Sigma\), whence
\begin{equation}
  \label{eq:t-prod-missing}
 t_{j_{1}}\cdots t_{j_{k}} = 0.
\end{equation}
Note that the order of the generators is irrelevant since they anti-commute.

If the characteristic of~\(\kk\) is \(2\), we recover the usual Stanley--Reisner algebra~\(\kk[\Sigma]\)
on commuting generators of degree~\(1\) and without differential, as remarked already in the introduction.

\section{Real toric spaces}
\label{sec:rts}

Let \(\Sigma\) be a simplicial complex on the set~\([m]\), and let
\begin{equation}
  Z_{\Sigma} = \ZK_{\Sigma}(D^{1},S^{0})
\end{equation}
be the associated real moment-angle complex. The \(2\)-torus~\(G=(\Z_{2})^{m}\) acts on~\(Z_{\Sigma}\)
in a canonical fashion. Let \(K\subset G\) be a freely acting subgroup.
The projection~\(Z_{\Sigma}\to X_{\Sigma}=Z_{\Sigma}/K\) is equivariant with respect to the quotient map~\(p\colon G\to L=G/K\).

\subsection{A dga model for real toric spaces}
\label{sec:dga-model}

The goal of this section is to prove our main result, announced already in \Cref{thm:main}.
We note that for \(2\)-tori there is no need to distinguish between left and right actions.

\begin{theorem}
  \label{thm:model-rts}
  The \(L\)-dga~\(A(\Sigma,L)\) is naturally quasi-isomorphic to~\(C^{*}(X_{\Sigma})\).
\end{theorem}

The \(L\)-dga~\(A(\Sigma,L)\) has been defined in the introduction.
In the statement above, naturality refers to the inclusion of subcomplexes~\(\Sigma'\hookrightarrow\Sigma\).

To prove \Cref{thm:model-rts} 
we construct a natural zigzag of quasi-iso\-mor\-phisms connecting the two dgas.
Instead of the additive group~\(\Z_{2}\) we use the multiplicative group~\(\CC=\{1,\cc\}\),
and we write \(\cc_{1}\),~\dots,~\(\cc_{m}\) for the canonical generators of~\(G=\CC^{m}\).

\subsubsection{First step}

We start by comparing the Borel constructions of~\(Z_{\Sigma}\) and~\(X_{\Sigma}\).

\begin{lemma}
  The map~\((Z_{\Sigma})_{G}\to (X_{\Sigma})_{L}\) induced by the projection~\(Z_{\Sigma}\to X_{\Sigma}\)
  is a homotopy equivalence and natural in~\(\Sigma\).
\end{lemma}

\begin{proof}
  Naturality is clear. Because \(K\) acts freely on~\(Z_{\Sigma}\) with quotient~\(X_{\Sigma}\), the map
  \begin{equation}
    (Z_{\Sigma})_{G} = EG\timesunder{G} Z_{\Sigma} \to EL\timesunder{L} X_{\Sigma} = (X_{\Sigma})_{L}
  \end{equation}
  is a bundle with fibre~\(EK\) and therefore a homotopy equivalence.
\end{proof}

Let
\begin{equation}
  \DJ_{\Sigma} = \ZK_{\Sigma}(B\CC,*) \subset BG
\end{equation}
be the simplicial version of the real Davis--Januszkiewicz space associated to~\(\Sigma\).
Depending on the context,
we think of it as a space over~\(BG\) or over~\(BL\), that is, together with the canonical map~\(\DJ_{\Sigma} \to BG\)
or its prolongation to~\(BL\) determined by the map~\(Bp\colon BG\to BL\).

\begin{lemma}
  \label{thm:homotopy-DJ}
  There is a zigzag of homotopy equivalences between~\((Z_{\Sigma})_{G}\) and~\(\DJ_{\Sigma}\).
  The zigzag consists of maps over~\(BG\) and is natural in~\(\Sigma\).
\end{lemma}

\begin{proof}
  The zigzag is
  \begin{equation}
    (Z_{\Sigma})_{G} = \ZK_{\Sigma} \bigl( E\CC\timesunder{\CC}D^{1},E\CC\timesunder{\CC}S^{0} \bigr)
    \leftarrow \ZK_{\Sigma} \bigl( E\CC\timesunder{\CC}D^{1},* \bigr)
    \rightarrow \ZK_{\Sigma}(B\CC,*),
  \end{equation}
  compare~\cite[Lemma~5.1]{Franz:torprod}.
  All maps are compatible with the maps to~\(BG=\ZK_{\Sigma}(B\CC,B\CC)\) and natural in~\(\Sigma\).
  (The map~\(*\to B\CC\) is the inclusion of the unique base point~\(\bar w_{0}\).)

  By induction on the size of~\(\Sigma\) it follows from the Künneth and Mayer--Vietoris theorems that each map is a quasi-iso\-mor\-phism.
  An analogous argument based on the Seifert--van Kampen theorem shows that each map induces an isomorphism
  on fundamental groups and therefore is a homotopy equivalence.
\end{proof}

Let \(Y_{\Sigma,L}\) be the pull-back of the universal bundle~\(EL\to BL\) along the composition~\(\DJ_{\Sigma}\hookrightarrow BG\to BL\).

\begin{lemma}
  \label{thm:X-Y}
  There is a zigzag of \(L\)-equivariant maps connecting \(X_{\Sigma}\) and~\(Y_{\Sigma,L}=\DJ_{\Sigma}\timesover{BL}EL\).
  It consists of homotopy equivalences of spaces and is natural in~\(\Sigma\).
\end{lemma}

\begin{proof}
  This is a consequence of \Cref{eq:equiv-koszul}, the two previous lemmas and the long exact sequence of homotopy groups of a bundle,
  \cf~\cite[Thm.~7.6, Prop.~18.4]{May:1968}.
\end{proof}

\subsubsection{Second step}
\label{sec:step-2}

Motivated by \Cref{thm:X-Y}, we study the spaces~\(Y_{\Sigma,L}\) in more detail,
in particular the special case~\(Y_{\Sigma}=Y_{\Sigma,G}\).

The group~\(G\) acts freely on~\(Y_{\Sigma}\). For any~\(\sigma\in\Sigma\) we have
\begin{equation}
  \label{eq:Y-sigma}
  Y_{\sigma} = \ZK_{\Sigma}(E\CC,\CC) = (E\CC)^{\sigma} \times \CC^{[m]\setminus\sigma}
  = \bigtimes_{j=1}^{m} Y_{\restr{\sigma}{j}}.
\end{equation}
Here we have written \(\restr{\sigma}{j}\) for the restriction of a simplex~\(\sigma\in\Sigma\)
to the vertex~\(j\in[m]\), which contains either the empty simplex only
or additionally the \(0\)-simplex~\(\{j\}\). Accordingly, \(Y_{\restr{\sigma}{j}}\) equals \(\CC\) or~\(E\CC\).
The decomposition~\eqref{eq:Y-sigma} is equivariant if we let the \(j\)-th factor of~\(G=\CC^{m}\)
act on the corresponding factor~\(Y_{\restr{\sigma}{j}}\).

We introduce several dgcs related to~\(\Sigma\).
For~\(\sigma\in\Sigma\) we define the \(G\)-dgc
\begin{align}
  \label{eq:def-M-sigma}
  M(\sigma) &= \bigotimes_{j=1}^{m}\,C(Y_{\restr{\sigma}{j}})
  \shortintertext{and based on this the \(G\)-dgc}
  \label{eq:def-M-Sigma}
  M(\Sigma) &= \colim_{\sigma\in\Sigma} M(\sigma)
  \shortintertext{as well as the \(L\)-dgc}
  M(\Sigma,L) &= M(\Sigma)/K.
\end{align}

Let us describe these objects explicitly. Forgetting the dgc structure, we have an isomorphism of graded \(L\)-modules
\begin{equation}
  \label{eq:def-M}
  M(\Sigma,L) = \SRC{\Sigma} \otimes \kk[L]
\end{equation}
where \(\kk[L]\) denotes the group algebra of~\(L\).
To see this, consider first the case~\(L=G\).
For~\(m=1\) and~\(\sigma=\emptyset\) our claim is clear. For~\(m=1\) and~\(\sigma=\{1\}\)
it follows from the canonical \(G\)-equivariant isomorphisms
\begin{equation}
  C(E\CC) = C(B\CC)\otimes\kk[\CC] = \SRC{\sigma}\otimes\kk[\CC],
\end{equation}
compare the discussion in \Cref{sec:universal}.
When passing to~\(m>1\), we first look again at a simplex~\(\sigma\in\Sigma\),
where we take tensor products as in~\eqref{eq:def-SRC-sigma} and~\eqref{eq:def-M-sigma} and reorder the factors.
For arbitrary~\(\Sigma\) we take colimits as in~\eqref{eq:def-SRC-Sigma} and~\eqref{eq:def-M-Sigma}
to obtain a \(G\)-equivariant isomorphism
\begin{equation}
  M(\Sigma) = \SRC{\Sigma} \otimes \kk[G]
\end{equation}
where on the right-hand side \(G\) acts on~\(\kk[G]\). For general~\(L\) we finally divide by~\(K\).

The induced differential on the right-hand side of~\eqref{eq:def-M} is given by
\begin{equation}
  \label{eq:diff-M}
  d(u_{\alpha}\otimes g) = d\,u_{\alpha}\otimes g
  + \sum_{\alpha_{j}>0} (-1)^{\epsilon(\alpha,j-1)}\,u_{\alpha|j} \otimes \bigl(1+ (-1)^{\alpha_{j}}\,\posub{j}\bigr)\,g
\end{equation}
where \(\epsilon(\alpha,j)\) has been defined in~\eqref{eq:def-epsilon-alpha-j},
and \(\posub{1}\),~\dots,~\(\posub{m}\) are the images of the canonical generators of~\(G=\CC^{m}\)
under the projection~\(p\colon G\to L\).
The diagonal is given by
\begin{equation}
  \label{eq:diag-M}
  \Delta(u_{\alpha}\otimes g) = \sum_{\beta+\gamma=\alpha} (-1)^{\epsilon(\beta,\gamma)}\,u_{\beta}\otimes \posup{\gamma}\,g
\end{equation}
for any~\(\alpha\in\N^{m}\) with~\(\supp\alpha\in\Sigma\) and~\(g\in G\),
where \(\epsilon(\beta,\gamma)\) is defined in~\eqref{eq:def-epsilon-beta-gamma} and
\begin{equation}
  \posup{\gamma} \coloneqq \posub{1}^{\gamma_{1}}\cdots \posub{m}^{\gamma_{m}}\in L.
\end{equation}
Later on we will only need the special cases
\begin{align}
  \label{eq:d-uj-g}
  d(u_{j}\otimes g) &= 1\otimes (1-\posub{j})\,g, \\
  \label{eq:Delta-1-g}
  \Delta(1\otimes g) &= (1\otimes g) \otimes (1\otimes g), \\
  \label{eq:Delta-uj-g}
  \Delta(u_{j}\otimes g) &= (u_{j}\otimes g) \otimes (1\otimes g) + (1\otimes \posub{j}\,g) \otimes (u_{j}\otimes g)
\end{align}
for~\(1\le j\le m\).

\begin{proposition}
  \label{thm:M-Y}
  There is a homotopy equivalence of \(L\)-complexes
  \begin{equation*}
    \phi_{\Sigma,L}\colon M(\Sigma,L)\to C(Y_{\Sigma,L})
  \end{equation*}
  that is a dgc map and natural in~\(\Sigma\) as well as~\(L\).
\end{proposition}

\begin{proof}
  In the case~\(L=G\)
  it follows from the identity~\eqref{eq:Y-sigma} and the Eilenberg--Zilber theorem
  that for any~\(\sigma\in\Sigma\) the shuffle map
  \begin{equation}
    \phi_{\sigma} = \shuffle\colon M(\sigma) = \bigotimes_{j=1}^{m}\, C(Y_{\restr{\sigma}{j}}) \to C(Y_{\sigma})
  \end{equation}  
  is a homotopy equivalence, and it is a map of coalgebras.
  Moreover, the naturality of the Eilenberg--Zilber contraction implies that the homotopy equivalence
  is natural with respect to inclusion of faces~\(\tau\hookrightarrow\sigma\)
  and also equivariant with respect to the action of the finite group~\(G=\CC^{m}\).
  Hence our claims hold if \(\Sigma\) is a simplex.
  
  The case of general~\(\Sigma\) now follows by induction on the number of simplices in~\(\Sigma\),
  using the diagram with exact rows
  \begin{equation}
    \begin{tikzcd}[column sep=small]
      0 \arrow{r} & M(\Sigma'\cap\Sigma'') \arrow[leftrightarrow]{d}{\simeq} \arrow{r} & M(\Sigma')\oplus M(\Sigma'') \arrow[leftrightarrow]{d}{\simeq} \arrow{r} &  M(\Sigma'\cup\Sigma'') \arrow[leftrightarrow]{d} \arrow{r} & 0 \\
      0 \arrow{r} & C(Y_{\Sigma'\cap\Sigma''}) \arrow{r} & C(Y_{\Sigma'})\oplus C(Y_{\Sigma'}) \arrow{r} & C(Y_{\Sigma'\cup\Sigma''}) \arrow{r} & 0 
    \end{tikzcd}
  \end{equation}
  to extend the homotopy equivalence to larger subcomplexes of~\(\Sigma\).

  The \(G\)-homotopy equivalence~\(\phi_{\Sigma}\) thus obtained
  descends to an \(L\)-homotopy equivalence
  \begin{equation}
    \phi_{\Sigma,L}\colon M(\Sigma,L) = M(\Sigma)/K \to C(Y_{\Sigma})/K = C(Y_{\Sigma}/K) = C(Y_{\Sigma,L}),
  \end{equation}
  which is again a morphism of dgcs and natural in~\(\Sigma\).
  Naturality with respect to~\(L\) is clear by construction.
\end{proof}

As a result, we see that the \(L\)-dga
\begin{equation}
  \BB(\Sigma,L)=M(\Sigma,L)^{*}
\end{equation}
dual to the \(L\)-dgc~\(M(\Sigma,L)\)
is naturally isomorphic to~\(C^{*}(Y_{\Sigma,L})\), hence naturally quasi-isomorphic to~\(C^{*}(X_{\Sigma})\) by \Cref{thm:X-Y}.

\begin{corollary}
  Then there is a natural quasi-iso\-mor\-phism of dgas
  \begin{equation*}
    C^{*}(\DJ_{\Sigma}) \to \kk[\Sigma].
  \end{equation*}
  In particular, if the characteristic of~\(\kk\) is \(2\), then \(\DJ_{\Sigma}\) is formal over~\(\kk\).
\end{corollary}

This last part is of course analogous to the integral formality of the (complex) Davis--Januszkiewicz space~\(\ZK_{\Sigma}(BS^{1},*)\)
established by the author~\cite[Thm.~3.3.2]{Franz:2001}, \cite[Thm.~1.4]{Franz:2006} and Notbohm--Ray~\cite[Thm.~4.8]{NotbohmRay:2005}.

\begin{proof}
  This is the special case~\(L=1\) of \Cref{thm:M-Y}: We have \(Y_{\Sigma,1}=\DJ_{\Sigma}\)
  and~\(M(\Sigma,1)=\SRC{\Sigma}\), hence also \(\BB(\Sigma,1)=\kk[\Sigma]\).
  
  If the characteristic of~\(\kk\) is \(2\), then the face algebra~\(\kk[\Sigma]\)
  is the usual Stanley--Reisner algebra of~\(\Sigma\)
  with commuting generators of degree~\(1\) and without differential,
  which therefore is the cohomology of~\(\DJ_{\Sigma}\).
\end{proof}

\subsubsection{Third step}

To complete the proof of \Cref{thm:main}, we show that \(\BB(\Sigma,L)\)
is naturally isomorphic to the \(L\)-dga~\(A(\Sigma,L)=\kk[\Sigma]\otimes\FU{L}\)
with product and differential as described in the introduction.
Recall that \(\FU{L}\) denotes the \(L\)-algebra of function~\(L\to\kk\).

Dualizing the isomorphism~\eqref{eq:def-M}, we obtain an isomorphism of graded \(L\)-modules
\begin{equation}
  \BB(\Sigma,L) = \kk[\Sigma] \otimes \FU{L} = A(\Sigma,L),
\end{equation}
natural in~\(\Sigma\). We have
\begin{equation}
  \pair{t^{\alpha}\otimes f,u_{\beta}\otimes g} =
  \begin{cases}
    f(g) & \text{if \(\alpha=\beta\),} \\
    0 & \text{otherwise}
  \end{cases}
\end{equation}
for all multi-indices~\(\alpha\),~\(\beta\) with support in~\(\Sigma\), all~\(g\in L\) and~\(f\in\FU{L}\).
It remains to show that the product~\(*\) and the differential on~\(\BB(\Sigma,L)\) agree with those of~\(A(\Sigma,L)\).
Motivated by the identity~\eqref{eq:prod-tj-f} below, we also write
\(t^{\alpha}\,f=t^{\alpha}\otimes f\in\BB(\Sigma,L)\).

We observe that the surjection of dgcs
\begin{equation}
  M(\Sigma,L) \to M(\Sigma,1) = \SRC{\Sigma}
\end{equation}
induced by the projection~\(L\to 1\) dualizes to an injection of dgas
\begin{equation}
  \kk[\Sigma] = \BB(\Sigma,1) \hookrightarrow \BB(\Sigma,L).
\end{equation}
Each generator~\(t_{j}\in\kk[\Sigma]\) pulls back to~\(t_{j}\in\BB(\Sigma,L)\) under this map.
Hence \(\kk[\Sigma]\) is a sub-dga of~\(\BB(\Sigma,L)\),
and the relations~\eqref{eq:diff-tj},~\eqref{eq:tj-tjp} and~\eqref{eq:t-prod-missing}
hold in~\(\BB(\Sigma,L)\).

From the formulas~\eqref{eq:Delta-1-g} and~\eqref{eq:Delta-uj-g} we can infer
\begin{equation}
  \label{eq:prod-tj-f}
  t_{j} * f = t_{j}\,f
  \qquad\text{and}\qquad
  f*h = f\,h
\end{equation}
for~\(1\le j\le m\) and~\(f\),~\(h\in\FU{L}\).
In particular, \(\FU{L}\) is a subalgebra of~\(\BB(\Sigma,L)\).

It remains to verify the formulas for the differential of the generators~\(s_{i}\)
and for the commutation relations among the~\(t_{j}\)'s and~\(s_{i}\)'s.
We observe the following:
Multiplying some element of~\(L\cong(\Z_{2})^{n}\) by~\(\posub{j}\) changes the \(i\)-th coordinate
if and only if \(x_{ij}=1\). Hence \(s_{i}\cdot\posub{j} = 1-s_{i}\) if \(x_{ij}=1\)
and \(s_{i}\cdot\posub{j} = s_{i}\) otherwise.
The commutation relation~\eqref{eq:rel-si-tj} can therefore be expressed as
\begin{equation}
  \label{eq:rel-f-tj-poj}
  f\,t_{j} = t_{j}\,f\cdot\posub{j}
\end{equation}
for all~\(1\le j\le m\) and~\(f=s_{i}\), hence for all elements~\(f\) of the \(L\)-algebra~\(\FU{L}\).

From~\eqref{eq:Delta-uj-g} we deduce
\begin{equation}
  \pair{f * t_{j},u_{j'}\otimes g}
  = \bigpair{f\otimes t_{j}, \Delta(u_{j'}\otimes g)}
  = \begin{cases}
    f(\posub{j}\,g) & \text{if \(j= j'\),} \\
    0 & \text{otherwise.}
  \end{cases}
\end{equation}
Because this is equivalent to~\eqref{eq:rel-f-tj-poj},
the products on~\(A(\Sigma,L)\) and~\(\BB(\Sigma,L)\) agree.

Moreover, the identity~\eqref{eq:d-uj-g} implies
\begin{align}
  \pair{d\,f,u_{j}\otimes g} &= - \pair{f,d(u_{j}\otimes g)}
  = - \bigpair{1\otimes f, 1\otimes (1-\posub{j})\,g} \\
  \notag &= f\bigl(\posub{j}\,g-g\bigr)
  = \bigpair{t_{j}\,f\cdot(\posub{j} - 1),u_{j}\otimes g}
\end{align}
for any~\(f\in\FU{L}\) and~\(1\le j\le m\), which shows
\begin{equation}
  d\,f = \sum_{j=1}^{m} t_{j}\,f\cdot(\posub{j} - 1).
\end{equation}
Combined with~\eqref{eq:rel-f-tj-poj} this gives the second formula in~\eqref{eq:intro:diff-si-tj}
and completes the proof of \Cref{thm:model-rts}.

\begin{remark}
  In the case of a real moment-angle complex~\(X_{\Sigma}=Z_{\Sigma}\) itself
  we can choose the characteristic matrix~\((x_{ij})\) to be the identity matrix.
  Let the dga~\(B(\Sigma)\) be the quotient of~\(A(\Sigma,G)\) by the additional relations
  \begin{equation}
    s_{i}\,t_{i} = t_{i}\,t_{i} = 0
  \end{equation}
  for~\(1\le i\le n\). We claim that the projection~\(A(\Sigma,G)\to B(\Sigma)\) is a quasi-iso\-mor\-phism.
  For~\(m=1\) this is checked directly. The Künneth theorem then proves the case where \(\Sigma\) is a simplex.
  For general~\(\Sigma\) one does induction over the size of~\(\Sigma\) as in the proof of \Cref{thm:M-Y}.
  Up to the sign in the formula~\(d\,s_{i}=-t_{i}\), we therefore recover Cai's description of the cohomology ring
  of a real moment-angle complex \cite[p.~514]{Cai:2017}.
  The minus sign can be eliminated by replacing each~\(t_{i}\) with~\(\tilde{t}_{i}=-t_{i}\).
\end{remark}

\subsection{Expressing the cohomology as a torsion product}

In this section we assume that the coefficient ring~\(\kk\) is of characteristic~\(2\).

The chosen decomposition~\(L\cong(\Z_{2})^{n}\) gives rise to an
isomorphism of graded algebras~\(R=H^{*}(BL)\cong\kk[\nn{t}_{1},\dots,\nn{t}_{n}]\).
The map~\(Bp^{*}\colon H^{*}(BL)\to H^{*}(BG)\) is of the form
\begin{equation}
  \nn{t}_{i} \mapsto \sum_{j=1}^{m} x_{ij}\,t_{j}.
\end{equation}
The Stanley--Reisner ring~\(\kk[\Sigma]\) is an algebra over~\(R\) via the map~\(Bp^{*}\).

\begin{proposition}
  \label{thm:iso-Tor}
  If the characteristic of~\(\kk\) is \(2\), then there is an isomorphism of graded vector spaces
  \begin{equation*}
    H^{*}(X_{\Sigma}) = \Tor_{R}(\kk,\kk[\Sigma]),
  \end{equation*}
  natural in~\(\Sigma\).
\end{proposition}

\begin{proof}
Let \(\nn{\Sigma}=[n]\) be the full simplex on \(n=m\)~vertices.
The differential graded \(R\)-module~\(\KK= A(\nn{\Sigma})\) is free over~\(R\) as a graded module
with basis
\begin{equation}
  s_{\sigma} = \prod_{i\in\sigma} s_{i}
\end{equation}
for~\(\sigma\subset[n]\). This together with the identity~\(d\,s_{i}=\nn{t}_{i}\) shows
that \(\KK\) is in fact the Koszul resolution of~\(\kk\) over~\(R\).
(Since \(X_{\nn{\Sigma}}=(D^{1})^{n}\) is contractible, \Cref{thm:model-rts}
confirms that \(\KK\) is acyclic.)

The differential~\eqref{eq:rel-mod2} of~\(s_{i}\) in~\(A(\Sigma,L)\) is
the image of~\(\nn{t}_{i}\) under the map~\(Bp^{*}\),
which means that we have an isomorphism of chain complexes
\begin{equation}
  A(\Sigma,L) = \KK \otimes_{R} \kk[\Sigma],
\end{equation}
natural in~\(\Sigma\).
Taking cohomology concludes the proof.
\end{proof}

\begin{remark}
  \label{rem:equivformal}
  Note that in general the canonical product on the \(\Tor\)~term does not correspond to the cup product.
  (This fails already for~\(\Sigma=\{\emptyset\}\).)
  It holds, however, if \(\kk[\Sigma]\) is free over~\(R\) because \(H_{G}^{*}(X_{\Sigma})=H^{*}(\DJ_{\Sigma})=\kk[\Sigma]\)
  surjects onto \(H^{*}(X_{\Sigma})=\Tor_{R}(\kk,\kk[\Sigma])=\kk\otimes_{R}\kk[\Sigma]\) in this case, compare~\cite[Prop.~7.1.6]{Hausmann:2014}.
  This happens for smooth real projective toric varieties or, more generally, for small covers with shellable~\(\Sigma\).
  We thus recover the description of their cohomology rings as given by Jurkiewicz~\cite[Thm.~4.3.1]{Jurkiewicz:1985}
  and Davis--Januszkiewicz~\cite[Thm.~4.14]{DavisJanuszkiewicz:1991}.
\end{remark}

\section{Toric varieties}
\label{sec:toric}

Let \(N\cong\Z^{n}\) be a lattice, and let \(\Sigma\) be a regular rational fan in~\(N_{\R}=N\otimes_{\Z}\R\).
As remarked in the introduction, the associated real toric variety~\(\XXX_{\Sigma}\) can be defined as the fixed point set
of the complex conjugation on the corresponding (complex) toric variety.
It can also be described intrinsically as toric varieties are defined over the integers,
compare~\cite[p.~78]{Fulton:1993} and~\cite[Sec.~2]{Franz:2010}.
In this final part of the paper we study \(\XXX_{\Sigma}\) from a topological point of view (equipped with the metric topology).
Recall that it comes with an action of the group~\(\LLL=\TTT_{N}\cong(\Ri)^{n}\) where \(\Ri=\R\setminus\{0\}\).

Let \(\tN\subset N\) be the sublattice spanned by the primitive generators of the rays in~\(\Sigma\).
We abbreviate \(N_{2}=N\otimes_{\Z}\Z_{2}\), and we write \(N_{2}/\tN_{2}\) for the quotient of~\(N_{2}\)
by the image of the map~\(\tN_{2}\to N_{2}\) induced by the inclusion. Recall that \(N_{2}\) is
canonically isomorphic to the group~\(L\cong(\Z_{2})^{n}\) contained in~\(\LLL\) and therefore also in~\(\XXX_{\Sigma}\).

The following two results have been established by Uma~\cite[Thm.~2.5\,(2), Rem.~7.4]{Uma:2004}
in the special case~\(\tN=N\).

\begin{lemma}
  \label{thm:components}
  The connected components of~\(\XXX_{\Sigma}\) are in bijection with the set~\(N_{2}/\tN_{2}\).
\end{lemma}

\begin{proof}
  We show that each component~\(W\subset\XXX_{\Sigma}\) contains a point~\(g\in L\)
  and that the assignment~\(W\mapsto [g]\in Q=N_{2}/\tN_{2}\) is well-defined and bijective.

  This is certainly true if~\(\Sigma=\sigma\) is a single cone, say of dimension~\(m\),
  since \(\XXX_{\Sigma}\cong\R^{m}\times(\Ri)^{n-m}\) and \(N_{2}/\tN_{2}\cong(\Z_{2})^{n-m}\) in this case.
  Note that if \(\tau\subset\sigma\) is a face and \(W\subset\XXX_{\sigma}\) a component,
  then all points in~\(W\cap\XXX_{\sigma}\cap L\) have the same image in~\(Q\).

  We now turn to the general case.
  Let \(\sigma\),~\(\sigma'\in\Sigma\), and
  let \(W\) and~\(W'\) be components of~\(\XXX_{\sigma}\) and~\(\XXX_{\sigma'}\), respectively.
  Assume that \(W\) and~\(W'\) intersect and let \(\tau=\sigma\cap\sigma'\).
  Then \(W\cap W'\) is a union of components of~\(\XXX_{\tau}\).
  By what we have said above, all points in~\(W\cap W'\cap L\)
  have the same image in~\(Q\), hence the same is true for all points in~\(W\cap L\) and~\(W'\cap L\).
  This proves that the map from the components of~\(\XXX_{\Sigma}\) to~\(Q\) is well-defined.
  It is also surjective by construction.

  Let~\(g\in L\), and let \(\bar y\in\tN_{2}\subset L\) be the image of the primitive generator~\(y\in\tN\) of some ray~\(\rho\) in~\(\Sigma\).
  Then \(g\) and~\(g+\bar y\) lie in the same component of~\(\XXX_{\rho}\), hence also in the same component~\(W\) of~\(\XXX_{\Sigma}\).
  By repeating this argument, we see that the map~\(W\mapsto[g]\) is injective.
\end{proof}

\begin{remark}
  \label{rem:splitting}
  Let \(N'=N\cap\tN_{\R}\) be the smallest reduced sublattice of~\(N\) containing all rays of~\(\Sigma\),
  and let \(\Sigma'\) be the fan~\(\Sigma\), considered as lying in~\(N'\).
  Writing \(\YYY_{\Sigma}=\XXX_{\Sigma'}\) and \(\LLL^{\Sigma}=\TTT_{N/N'}\), we have
  a canonical inclusion~\(\YYY_{\Sigma}\hookrightarrow\XXX_{\Sigma}\)
  and a non-canonical isomorphism~\(\XXX_{\Sigma} \cong \YYY_{\Sigma}\times \LLL^{\Sigma}\).
\end{remark}

Fix an ordering of the, say, \(m\)~rays in~\(\Sigma\), and define a linear map~\(\hN=\Z^{m}\to N\)
be sending each basis vector to the minimal representative of the corresponding ray in~\(\Sigma\).
Let \(\hSigma\) be the subfan of the positive orthant in~\(\hN_{\R}\)
that is combinatorially equivalent to~\(\Sigma\) under this map.
The associated real toric variety~\(\ZZZ_{\Sigma}=\XXX_{\hSigma}\) is the ``real Cox construction'';
it is the complement of the real coordinate subspace arrangement defined by~\(\Sigma\), considered as a simplicial complex.
(See~\cite[Thm.~5.4.5]{BuchstaberPanov:2015} or~\cite[Sec.~4]{Franz:torprod} for the analogous construction for toric varieties.)
By construction, the Cox construction
comes with a map of fans \((\hSigma,\hN)\to(\Sigma,N)\),
hence with a morphism of real toric varieties~\(\ZZZ_{\Sigma}\to\XXX_{\Sigma}\).
Let \(\KKK\) be the kernel of the associated map of real algebraic tori~\(\GGG=\TTT_{\hN}\to\LLL=\TTT_{N}\).

The following example illustrates that the real Cox construction
behaves differently from the complex case in that it may fail to surject onto~\(\XXX_{\Sigma}\)
even if the rays in~\(\Sigma\) span the vector space \(N_{\R}\).

\begin{example}
  Consider the fan in~\(\Z^{2}\) with rays spanned by the vectors~\(y_{1}=[1,0]\) and~\(y_{2}=[-1,2]\)
  (and no \(2\)-dimensional cones).
  As discussed in~\cite[Sec.~9]{Franz:torprod}, the corresponding toric variety
  is \(\XXX_{\Sigma}(\C)=(\C^{2}\setminus\{0\})/\pm1\simeq S^{3}/\pm1=\RP^{3}\).
  This description in fact is the quotient of the Cox construction.

  The real Cox construction therefore is \(\R^{2}\setminus\{0\}\), which consists
  of two copies of~\(\R\times\Ri\simeq D^{1}\times S^{0}\) that are glued together
  such that the compact retracts form the edges of a square.
  The real toric variety~\(\XXX_{\Sigma}=\XXX_{\Sigma}(\R)\) itself is also covered
  by two copies of~\(\R\times\Ri\). Because \(y_{1}\) and~\(y_{2}\) agree modulo~\(2\),
  the gluing map between the two copies of~\(\Ri\times\Ri\) is the identity, so that
  the compact retracts form two circles.
  The map~\(\ZZZ_{\Sigma}\to\XXX_{\Sigma}\) is again the quotient by~\(\KKK=\Z_{2}\),
  but this time it is not surjective.
\end{example}

The relation between a real toric variety and its Cox construction is as follows.
(Our formulation is in fact equally true in the complex setting.)

\begin{proposition}
  \label{thm:cox-real}
  As a topological space, \(\ZZZ_{\Sigma}\) is a principal \(\KKK\)-bundle, and
  there is an \(\LLL\)-equivariant homeomorphism
  \begin{equation*}
    \XXX_{\Sigma} \approx \LLL\timesunder{\GGG/\KKK}\ZZZ_{\Sigma}/\KKK.
  \end{equation*}
\end{proposition}

\begin{proof}
  Using \Cref{rem:splitting}, we can reduce the claim
  to the case where the rays of~\(\Sigma\) span~\(N_{\R}\), so that \(N\) and~\(\tN\) have the same rank.  
  Moreover, as mentioned earlier, the case~\(\tN=N\) of our claim appears in~\cite[Lemma~7.3, Rem.~7.4]{Uma:2004};
  we have \(\KKK\cong(\Ri)^{n-m}\) and \(\LLL=\GGG/\KKK\) in this case.
  It therefore suffices to consider the map
  \begin{equation}
    \pi\colon \tXXX_{\Sigma} \to \XXX_{\Sigma}
  \end{equation}
  where \(\tXXX_{\Sigma}=\ZZZ_{\Sigma}/\tKKK\) is the real toric variety associated to~\(\Sigma\),
  considered as a fan in~\(\tN\).
  Here \(\tKKK\) is the kernel of the quotient map~\(\GGG\to\tLLL=\TTT_{\tN}\).

  Both the kernel and cokernel of the map~\(\tLLL\to\LLL\) are isomorphic to~\(Q=N_{2}/\tN_{2}\),
  and its image is~\(\bGGG\coloneqq\LLL/Q=\GGG/\KKK\). Let \(q\) be the size of~\(Q\).
  Let us consider the restriction
  \begin{equation}
    \begin{tikzcd}
      \tYYY_{\!\!\sigma} \arrow[equal]{d} \arrow[hook]{r} & \tXXX_{\sigma} \arrow{d} \arrow[equal]{r} & \tYYY_{\!\!\sigma} \times \tLLL^{\sigma}  \arrow{d} \\
      \YYY_{\sigma} \arrow[hook]{r} & \XXX_{\sigma} \arrow[equal]{r} & \YYY_{\sigma} \times \LLL^{\sigma}
    \end{tikzcd}
  \end{equation}
  of~\(\pi\) for~\(\sigma\in\Sigma\). It is equivariant with respect to the map~\(\tLLL^{\sigma}\to\LLL^{\sigma}\),
  which again has kernel~\(Q\) since the regular cone~\(\sigma\) spans the same sublattice in~\(N\) and~\(\tN\).
  This implies that \(Q\) freely permutes the connected components of~\(\tXXX_{\sigma}\),
  and that its orbits are the non-empty fibres of the map~\(\pi\).

  \Cref{thm:components} tells us that \(\XXX_{\Sigma}\) has \(q\)~components,
  while \(\tXXX_{\Sigma}\) is connected, so that its image lies in the component~\(W\subset\XXX_{\Sigma}\) containing \(1\in L\).
  Fix a~\(\sigma\in\Sigma\), and let \(r\) be the number of components of~\(\XXX_{\sigma}\),
  which is equal to the corresponding number for~\(\tXXX_{\sigma}\).
  By equivariance, each component of~\(\XXX_{\Sigma}\) contains the same number~\(r/q\) of components of~\(\XXX_{\sigma}\).
  On the other hand, the image of the map~\(\pi_{\sigma}\) also has \(r/q\) components.
  This implies that \(W\) is the image of~\(\pi\) and therefore the quotient of~\(\tXXX_{\Sigma}\) by~\(Q\).
  
  As a result, we see that \(\ZZZ_{\Sigma}\) is a principal \(\KKK\)-bundle with base~\(W\).
  The projection map~\(\ZZZ_{\Sigma}\to W\) is equivariant with respect to the quotient map~\(\GGG\to\bGGG\).
  Since each of the \(q\)~components of~\(\XXX_{\Sigma}\) is of the form~\(g\cdot W\) for some~\(g\in L\),
  the surjective map~\(\LLL\times W\to\XXX_{\Sigma}\) induces the desired \(\LLL\)-equivariant homeomorphism.
\end{proof}

\begin{remark}
  \label{rem:natural}
  Let \(\sigma\in\Sigma\) with corresponding cone~\(\hsigma\in\hSigma\).
  The toric morphism \(\ZZZ_{\Sigma}\to\XXX_{\Sigma}\)
  sends the \(\GGG\)-orbit of~\(\ZZZ_{\Sigma}\) indexed by~\(\hsigma\)
  to the \(\LLL\)-orbit of~\(\XXX_{\Sigma}\) indexed by~\(\sigma\).
  The homeomorphism from \Cref{thm:cox-real} therefore is natural in the following sense:
  If we form the fan~\(\hSigma\) in~\(\hN\) for a given fan~\(\Sigma\),
  we can use subfans of it in the same lattice~\(\hN\)
  to construct the Cox constructions~\(\ZZZ_{\Sigma'}\) for all subfans~\(\Sigma'\subset\Sigma\).
  By what we have just said, the homeomorphism~\(\XXX_{\Sigma} \approx \LLL\timesunder{\GGG/\KKK}\ZZZ_{\Sigma}/\KKK\)
  then restricts to a homeomorphism~\(\XXX_{\Sigma'} \approx \LLL\timesunder{\GGG/\KKK}\ZZZ_{\Sigma'}/\KKK\).
\end{remark}

\begin{corollary}
  \label{thm:rts-toric-retraction}
  We keep the notation and set \(K=\KKK\cap G\).
  Then there is an \(L\)-equivariant strong deformation retract
  \begin{equation*}
    L \timesunder{G/K} Z_{\Sigma}/K \hookrightarrow \XXX_{\Sigma},
  \end{equation*}
  natural in~\(\Sigma\).
\end{corollary}

\begin{proof}
  We write \(\bG=G/K\).
  As mentioned in the introduction already, the inclusion \(Z_{\Sigma} \to \ZZZ_{\Sigma}\)
  is a \(G\)-equivariant strong deformation retract, see~\cite[Thm.~2.1]{Franz:2010}.
  By an argument analogous to~\cite[Prop.~4.1]{Franz:torprod} this implies
  that the map~\(Z_{\Sigma}/K \to \ZZZ_{\Sigma}/\KKK\)
  is a \(\bG\)-equivariant strong deformation retract. Hence
  \begin{equation}
    L \timesunder{\bG} Z_{\Sigma}/K \hookrightarrow L \timesunder{\bG} \ZZZ_{\Sigma}/\KKK
    = \LLL \timesunder{\GGG/\KKK} \ZZZ_{\Sigma}/\KKK \approx \XXX_{\Sigma},
  \end{equation}
  is an \(L\)-equivariant strong deformation retract.
  In the last step we have used \Cref{thm:cox-real} and before that
  the fact that \(\GGG/\KKK\) contains the connected component of the identity in~\(\LLL\).
  Naturality holds as explained in \Cref{rem:natural}.
\end{proof}

The \(L\)-dga~\(A(\Sigma,L)\) from the introduction is still well-defined for the real toric variety~\(\XXX_{\Sigma}\)
instead of a real toric space. We keep the ordering of the rays in~\(\Sigma\) chosen for the Cox construction.
The coefficient~\(x_{ij}\) of the characteristic matrix then is the \(i\)-th coordinate (modulo~\(2\)) of the primitive generator~\(y_{j}\) of the \(j\)-th ray.

\begin{theorem}
  \label{thm:model-toric}
  The \(L\)-dga~\(A(\Sigma,L)\) is naturally quasi-isomorphic to~\(C^{*}(\XXX_{\Sigma})\).
\end{theorem}

\begin{proof}
  As in \Cref{sec:dga-model}, we replace \(\XXX_{\Sigma}\) by the homotopy-equivalent space
  \begin{equation}
    Y_{\Sigma} = (\XXX_{\Sigma})_{L} \timesover{BL} EL.
  \end{equation}
  Writing \(\bG=G/K\) as in the previous proof, we have
  \begin{equation}
    (\XXX_{\Sigma})_{L} = EL \timesunder{L} \XXX_{\Sigma} \simeq E\bG \timesunder{\bG} Z_{\Sigma}/K \simeq \DJ_{\Sigma}
  \end{equation}
  by \Cref{thm:rts-toric-retraction} and \Cref{thm:homotopy-DJ},
  and we can complete the proof as from \Cref{sec:step-2} on.
\end{proof}

\begin{corollary}
  \label{thm:maximal}
  Smooth toric varieties are maximal.
\end{corollary}

\begin{proof}
  As a consequence of \Cref{thm:model-toric}, \Cref{thm:iso-Tor} (with~\(\kk=\Z_{2}\)) carries over to real toric varieties.
  Moreover, an analogous result is valid for smooth (complex) toric varieties (even for arbitrary coefficients),
  see~\cite[Thm.~3.3.2]{Franz:2001} or~\cite[Thm.~1.2]{Franz:2006}. Of course, \(R\) and~\(\Z_{2}[\Sigma]\)
  must be graded evenly in this case.
  Nevertheless, as ungraded \(\Z_{2}\)-vector spaces we have the same torsion product both for the
  real and the complex points of a smooth toric variety, hence the same mod~\(2\) Betti sum.
\end{proof}

\begin{remark}
  The notion of maximality makes sense for any space with an involution,
  in particular for the partial quotient~\(\ZK_{\Sigma}(D^{2},S^{1})/\KKKK\) of a moment-angle complex
  by a freely acting closed subgroup~\(\KKKK\) of the ambient torus~\(\GGGG=(S^{1})^{m}\).
  \Cref{thm:maximal} extends to this case if and only if the induced map~\(\GGGG\to\GGGG/\KKKK=\LLLL\)
  restricts to a surjection~\(G\to L\) on the \(2\)-torsion elements.
  This happens if and only if \(\KKKK\) is of the form~\(\tilde\KKKK\times Q\)
  where \(\tilde\KKKK\) is a subtorus and \(Q\) finite of odd order.
\end{remark}

\end{document}